\documentclass[letterpaper]{article}

\title{Highly nonrepetitive sequences: winning strategies from the Local Lemma}
\author{Wesley Pegden\footnote{
Department of Mathematics, 
Rutgers University (New Brunswick), 
110 Frelinghuysen Rd.,
Piscataway, NJ 08854-8019. 
Email: pegden@math.rutgers.edu}
}

\date{September 10, 2010}

\usepackage{amsmath,amsthm,amssymb,amscd,latexsym,pstricks,pst-node,pst-tree}
\usepackage{url}
\usepackage{cite}
\usepackage{mathtools}

\newcommand{\diffblock}[1]{#1}

\newcommand{\pp}{{\rm P}}

\newcommand{\aaa}{{\cal A}}
\newcommand{\ccc}{{\cal C}}
\newcommand{\bbb}{{\cal B}}

\newcommand{\sbs}{\subset}
\newcommand{\abs}[1]{\lvert #1 \rvert}
\newcommand{\flr}[1]{\left\lfloor #1 \right\rfloor}
\newcommand{\clg}[1]{\lceil #1 \rceil}

\newcommand{\st}{\,\vrule\,}

\newcommand{\Z}{\mathbb{Z}}
\newcommand{\N}{\mathbb{N}}

\newtheorem{theorem}{Theorem}[section]
\newtheorem*{theoremRinner}{Theorem~\ref{\repthmref} (restated)}
\newenvironment{theoremR}[1]
{\def\repthmref{#1}\theoremRinner}{\endtheoremRinner}

\newtheorem{conjecture}[theorem]{Conjecture}

\newtheorem{observ}[theorem]{Observation}

\newcommand{\comments}[1]{}

{
\theoremstyle{definition}

\newtheorem{q}{}
}

{
\theoremstyle{remark}

}

\newcommand{\stm}{\setminus}

\newcommand{\ep}{\varepsilon}
\newcommand{\eps}{\epsilon}

\begin{document}

\maketitle

\begin{abstract}
  We prove game-theoretic versions of several classical results on nonrepetitive sequences, showing the existence of winning strategies using an extension of the Lov\'asz Local Lemma which can dramatically reduce the number of edges needed in a dependency graph when there is an ordering underlying the significant dependencies of events.  This appears to represent the first successful application of a Local Lemma to games.
\end{abstract}

\section{Introduction}
In 1906, Thue showed the existence of infinite ternary square-free sequences---sequences without any adjacent identical blocks\cite{t1,t2}.  Note that there can be no binary sequence even just of length 4 with this property: after 010 or 101, there is no continuation which does not introduce a repetition.  To obtain Thue's remarkable ternary sequence, recursively define  binary sequences $T_i$ where $T_0=0$, and $T_{i+1}$ is constructed by replacing each $0$ in $T_{i}$ with the string $01$ and each 1 with the string $10$.  Thus $T_1=01$, $T_2=0110$, etc.  Each $T_i$ is the initial segment of $T_{i+1}$, so the limit $T_\infty =\lim_{i\to \infty} T_i$ is a well-defined binary sequence
\[
01 10 1001 10010110 1001011001101001 10010110011010010110100110010110\dots.
\]
Thue observed that $T_\infty$ is \emph{overlap-free}---no two overlapping intervals are identical---since this property is preserved under the replacement operation described above.  In particular, this implies that $T_\infty$ is cube-free: it contains no 3 consecutive identical blocks.  Construct the sequence
\[
2102012101202102012021012102012\dots
\]
by counting the number of 1's between each consecutive pair of 0's in $T_\infty$.  Since $T_\infty$ is cube-free, this is a ternary sequence.  The fact that $T_\infty$ is overlap-free implies that this sequence is square-free.

 This concept has many interesting generalizations and suggests several directions of research; see  \cite{cprobs} and \cite{survey} for overviews of the area and discussions of open problems.  One such direction concerns requiring that any identical blocks be far apart, rather than simply nonadjacent.  This is the subject of the following conjecture of Dejean\cite{dej}:
\begin{conjecture}[Dejean (1972)]
  For every $c\geq 5$, there is an infinite $c$-ary sequence where (for all $n$) any two identical blocks of length $n$ are separated by at least $(c-2)n$ terms.
\end{conjecture}
\noindent The conjecture is now proved.  After being confirmed for $c\leq 14$ \cite{d4,d5t11,d12t14} and $c\geq 30$ \cite{d30t32,d33up}, proofs of the remaining cases were announced by both Currie and Rampersad \cite{crD}, and Rao \cite{rD}.   If one is concerned only with large identical blocks, a much stronger restriction can be placed on the distance between identical pairs, even in a binary sequence:
\begin{theorem}[J. Beck (1981)]
  For any $\ep>0$, there is some $N_\ep$ and an infinite 0-1 sequence such that any two identical blocks of length $n>N_\ep$ are at distance greater than $(2-\ep)^n$.\qed
\label{t.beck}
\end{theorem}
Theorem \ref{t.beck}, which is essentially best possible ($(2-\ep)^n$ cannot be replaced by $2^n$), was an early application of the `General' (or `Asymmetric') Lov\'asz Local Lemma; the existence of the relevant 0-1 sequence is proved by showing that the probability that a random finite sequence has the properties in question is positive, and then using compactness to show the existence of a suitable infinite sequence.  Note that no explicit construction of such a sequence is known.

While sequences that have nonrepetitive properties like those guaranteed by Beck's Theorem \ref{t.beck} may be very scarce, we will show in this paper that they are nevertheless very `attainable', in a certain sense (although we do not know how!).  Consider a game, the \emph{binary sequence game}, in which two players take turns choosing from the digits $\{0,1\}$ to form an unending binary sequence (the first digit is chosen by Player 1, the second by Player 2, the third by Player 1, \emph{etc.}).  
Our first result is an analog to Beck's theorem, showing that we can construct highly nonrepetitive sequences even when faced with an adversary.
\begin{theorem}
  For any $\ep>0$, there is an $N_\ep$ for which Player 1 has a strategy in the binary sequence game which ensures that any two identical blocks of length $n>N_\ep$ in the resulting sequence will lie at distance greater than $(2-\ep)^{n/2}$.
\label{t.2game}
\end{theorem}
Roughly speaking, Beck's Theorem \ref{t.beck} asserts that we can build a binary sequence where any large blocks are separated by exponentially large distance.  Theorem \ref{t.2game} asserts that we can build such a sequence even if an adversary gets to choose every other digit.  Of course, as we have surrendered control over `half' the sequence, it is not surprising that the base of the exponent is lower.  We will see shortly that $(2-\ep)^{\frac n 2}$ is best possible here.

Surprisingly, the proof of Theorem \ref{t.2game} is also probabilistic; it is based on an extension of the Local Lemma (Theorem \ref{t.ol}), and is many steps removed from any constructive argument.  We use our extension of the Local Lemma to show (nonconstructively) that Player 2 cannot have a winning strategy in a finite version of the game; this implies (nonconstructively) that Player 1 has a winning strategy for any finite version of the game; this, finally, implies via compactness (nonconstructively) the existence of a winning strategy for Player 1 in the infinite version of the game.  Needless to say, Player 1's winning strategy may be impossible to determine explicitly---it seems it may not even have a finite description.

Player 1 can ensure exponential distance between long identical blocks even if he is an `underdog', playing a $(1:t)$ biased game where the second player makes $t$ moves between each move of Player 1.  In this case we have the following:
\begin{theorem}
  For any $\ep>0$, there is an $N_{\ep,t}$ for which Player 1 has a strategy in the $(1:t)$ biased binary sequence game to ensure that any two identical blocks of length $n>N_{\ep,t}$ in the resulting sequence will lie at distance greater than $(2-\ep)^{n/(t+1)}$.
\label{t.tgame}
\end{theorem}
Thus Player 1 can obtain a highly nonrepetitive sequence even in the face of a powerful adversary.  As remarked earlier, we know of no explicit sequence in which long identical blocks are exponentially far apart.  Nevertheless, Theorems \ref{t.2game} and \ref{t.tgame} can be interpreted as showing that very `robust' construction strategies for such sequences do exist.  

We note that Theorems \ref{t.2game} and \ref{t.tgame} are actually essentially best possible:  assume Player 1 has a strategy in the $(1:t)$ biased binary sequence game to force any long blocks of length $n>N$ to be at distance greater than $(2+\ep)^{n/(t+1)}$, and let Player 2 employ the simple strategy of always choosing the digit 0.  We let $S$ denote the 0-1 sequence produced by the game, and let $S'$ denote the sequence consisting of just Player 1's moves (so, it consists of every $(t+1)$st term from $S$).  Since all of Player 2's moves are 0, the fact that $S$ contains no identical blocks of length $n>N$ at distance greater than $(2+\ep)^{n/(t+1)}$ implies that $S'$ contains no identical blocks of length $n'=n/(t+1)$ at distance greater than $\frac 1 {t+1} (2+\ep)^{n/(t+1)}$.  For sufficiently large $n$ (depending on $\ep$) this implies that $S'$ contains no identical blocks of length $n'$ within distance $2^{n/t+1}=2^{n'}.$   Thus the contradiction follows from the same argument which shows that Beck's theorem is essentially best possible:  by the pigeonhole principle, there must be two identical blocks of length $n'$ whose first terms both lie among the first $2^{n'}$ terms of $S'$. Note that the same kind of reduction shows that that Theorems \ref{t.2game} and \ref{t.tgame} are in fact generalizations of Beck's Theorem \ref{t.beck} (even for \emph{any} fixed value of $t$, in the latter case).

Another result proved with the Local Lemma and related to Beck's theorem is the subject of Exercise 2 in Chapter 5 of the Alon-Spencer book\cite{pm}, which asserts that for any $\ep>0$ there is an $N_\ep$ and a $\{0,1\}$ sequence $\alpha_1,\alpha_2,\alpha_3,\dots$ in which any two adjacent blocks $\alpha_{k+1},\dots,\alpha_{k+n}$ and $\alpha_{k+n+1},\dots,\alpha_{2n}$, each of length $n>N_\ep$, differ in at least $(\frac 1 2-\ep)n$ places (so $\alpha_{k+j}\neq \alpha_{k+n+j}$ for at least $(\frac 1 2 -\ep)n$ values of $j$ between $1$ and $n$).  (It is not hard to check that this is essentially best possible; begin with the observation that for any $\{0,1\}$ coloring of $[M]$, at least half all pairs $\{\alpha_i,\alpha_j\}\sbs [M]$ will monochromatic.)   Using the same extension of the Local Lemma used to prove Theorem \ref{t.2game}, we prove the following game-theoretic version of this result.

\begin{theorem}
  For any $\ep>0$, there is an $N_{\ep,t}$ such that Player 1 has a strategy in the $(1:t)$ biased binary sequence game to ensure that any two adjacent blocks of length $n>N_{\ep,t}$ differ in at least $\left(\frac 1 {2t+2}-\ep\right)n$ places.
\label{t.agame}
\end{theorem}
As in the case of Theorems \ref{t.2game} and  \ref{t.tgame}, one can check that Theorem \ref{t.agame} is essentially best possible.

A natural game-based question related to Thue-type sequences is whether there is an analogous theorem to Thue's original result on the nonrepetitive ternary sequence: namely, is there some base $c$ such that, in the $c$-ary sequence game, Player 1 can prevent the occurrence of \emph{any} identical adjacent blocks (even short ones)?  Of course, by simply imitating Player 1's moves, Player 2 can create lots of identical adjacent blocks of length 1.  In Section \ref{s.rainbow}, we prove that this is the best Player 2 can hope to achieve assuming $c$ is sufficiently large.  In fact, we will show (Theorem \ref{t.grain}) that for any $k\geq 2$ and any $c$ sufficiently large (depending on $k$), Player 1 has a strategy in the $c$-ary sequence game which ensures that out of any $k$ consecutive blocks all of the same length $\geq 2$, no two are identical.  This is a game-theoretic version of a result of Grytczuk on \emph{$k$-nonrepetitive} sequences proved using the Local Lemma \cite{rainbow}. For the most natural case $k=2$, we get the result for $c=37$ by optimizing the application of our Lefthanded Local Lemma.

A natural extension of the concept of nonrepetitiveness is given by pattern avoidance.  A word \emph{matches} a pattern $p_1p_2\cdots p_k$ if it can be partitioned into consecutive blocks $B_i$ such that $p_i=p_j$ implies that $B_i=B_j$.  To \emph{avoid} a pattern in a sequence (or string) means that there is no matching subword---thus, a square-free sequence is just one avoiding the pattern $xx$.

Unavoidable patterns are those which must appear in any sequence over a finite number of symbols.  Unavoidable patterns were characterized in \cite{bean} and \cite{Zunav}, and are exactly those patterns which match any of the words in the set $\{x,xyx,xyxzxyx,xyxzxyxwxyxzxyx,\dots\}$.   The \emph{index} of an avoidable pattern is the smallest base for which there is an infinite sequence to that base avoiding it.  Thus the index of $xx$ is 3, while the index of $xxx$ is 2.  Surprisingly, in spite of the characterization of unavoidable patterns, there are no avoidable patterns known to have large index, and it is open question whether there may be some absolute bound on the index of avoidable patterns---maybe even 6 is an upper bound.  Call a variable in a pattern \emph{isolated} if it occurs exactly once.  Note that the characterization of unavoidable patterns implies that a pattern without isolated variables is avoidable.  A result of J. Grytczuk on pattern avoidance in graphs \cite{patgraph} implies that such patterns actually have avoidability index at most 119.  Our extension of the Local Lemma improves this bound to 22, but an algebraic proof has recently been found which gives a bound of 4\cite{bgoh}.  In Section \ref{s.patt} we prove the following game-theoretic analog of this result:
\begin{theorem}
For any pattern $p$ with no isolated variables, Player 1 has a strategy in the 429-ary sequence game to ensure that the sequence that results from game-play does not contain a word which matches $p$ under a partition consisting of blocks $B_i$ all of lengths $\abs {B_i}\geq 2$.
\label{t.gpattern}
\end{theorem}
Note that, for Theorems \ref{t.2game}, \ref{t.agame}, \ref{t.grain}, and (for many patterns) \ref{t.gpattern}, the strategies which are guaranteed to exist for Player 1 must really be adaptive strategies that depend on the previous moves of Player 2. Note for example that if Player 2 was allowed to know Player 1's moves in advance, he could ensure the existence of pairs of arbitrarily long identical adjacent blocks: for any pair of adjacent blocks of odd length, he could make his moves in the first block to match the moves of Player 1 in the second one, and his moves in the second block to match Player 1's moves in the first block.

Our results in this paper represent the first successful application of a  Local Lemma to games.  There are many natural games where there is a strong Local-Lemma based probabilistic intuition suggesting the existence (or not) of a winning strategy for a player.  Previous results, however,  have come from `derandomizing' the intuition to give constructive proofs, and in some cases, the intuition suggested by the Local Lemma remains unproven (see \emph{e.g.} \cite{bsur} for a discussion).  The Lefthanded Local Lemma we present also allows improvements (for example in the $\ep$'s in Theorems \ref{t.beck} and the Alon-Spencer exercise) to previous game-free Local Lemma arguments on sequences.  Grytczuk, Przyby\l{}o, and Zhu \cite{Tchoice} have recently used restricted sampling techniques together with the Lefthanded Local Lemma to give near-optimum bounds for the `Thue choice number', the `list-chromatic' analog to the number of colors (3) required in a square-free sequence.  Their upper bound of 4 is just 1 more than the immediate lower bound of 3.

\section{An easier game}
\label{s.easy}
In this section we prove a weaker version of Theorem \ref{t.2game} to clarify the role of the Local Lemma in this type of problem, and show how the need for our `ordered' version of the Lemma arises.  The proof in this section is very similar to Beck's proof of Theorem \ref{t.beck}.
\begin{theorem}
  For any $\ep>0$ there is an $N_\ep$ such that for any $\{0,1\}$-assignment of the variables $d_2,d_4,d_6,\dots$, there is a $\{0,1\}$-assignment of the variables $\eps_1,\eps_3,\eps_5,\dots$ for which the sequence $\eps_1d_2\eps_3d_4\dots$ contains no identical intervals of a length $n>N_\ep$ and at distance $<(2-\ep)^{\frac n 2}$.
\label{t.easy}
\end{theorem}
Compared with Theorem \ref{t.2game}, Theorem \ref{t.easy} addresses the situation where Player 1 knows all of Player 2's moves in advance (and so Player 2's moves cannot depend on any of Player 1's moves).   The proof of Theorem \ref{t.easy} uses the Local Lemma.  We recall the following general version of the lemma, due to Erd\H{o}s and Spencer\cite{ES}, sometimes referred to as the `Lopsided' Local Lemma.

\begin{theorem}[\textbf{Lopsided General Local Lemma}]
Consider a finite family of events $\aaa$ as the vertex-set of a directed graph $G$.  Suppose that there are real numbers $0<x_A<1$ ($A\in \aaa$) such that, for each $A\in \aaa$ and each $\ccc\subset \aaa\stm \Gamma(A)$,
\label{t.lll}
\begin{equation}
\pp\left(A|\bigcap_{C\in\ccc}\bar C\right)\leq x_A\prod_{B\gets A} (1-x_B).
\label{l.lll}
\end{equation}
Then $\pp(\bigcap_{A\in \aaa} \bar A)>0$.\qed
\end{theorem}
\noindent Here $B\gets A$ means that $(A,B)$ is an edge of $G$ ($B$ is an \emph{out-neighbor} of $A$), and $\Gamma(A)$ denotes the out-neighbor set of $A$.  Observe that in the case where $A$ is independent of all collections of events $C\not\gets A$, line (\ref{l.lll}) is satisfied so long as the required bound holds for the unconditional probability $P(A)$; this gives the normal statement of the Local Lemma, which is sufficient for our proof in this section.

To apply the Local Lemma, we need to work with a finite family of bad events.  For this purpose we prove a finite version of the theorem, from which Theorem \ref{t.easy} follows by compactness.  For the sake of convenience, set $\alpha_i=\eps_i$ for odd $i$, and $\alpha_i=d_i$ for even $i$; so $\alpha_i$ is just the $i$th term of the sequence referred to in Theorem \ref{t.easy}.

\begin{theorem}
  For any $\ep>0$ there is an $N_\ep$ such that for any $M$ and any $\{0,1\}$-assignment of the variables $d_2,d_4,\dots,d_{(2\flr{\frac M 2})}$, there is a $\{0,1\}$-assignment of the variables $\eps_1,\eps_3,\dots,\eps_{(2\flr{\frac {M-1}{2}}+1)}$ so that the sequence $\eps_1d_2\eps_3d_4\dots \alpha_M$ contains no pair of identical intervals of length $n>N_\ep$ and at distance $<(2-\ep)^{\frac n 2}$.
\label{t.feasy}
\end{theorem}
\begin{proof}
Fix $\ep>0$ and any $\{0,1\}$-assignment of the variables $d_2,d_4,\dots,d_{(2\flr{\frac M 2})}$ (Player 2's moves).  We let Player 1's moves (the $\eps_i$'s)  be chosen by independent random coin-flips: each variable is 1 or 0 with probability $\frac 1 2$ for either case.  Let $f(n)=(2-\ep)^{\frac n 2}$: this is the distance we seek to guarantee that any pairs of identical intervals of length $n>N_\ep$ can be separated by.

To apply the Local Lemma, we let the event $A_{k,\ell,n}$ $(k<\ell)$ indicate that the blocks $\alpha_{k+1},\alpha_{k+2},\dots,\alpha_{k+n}$ and $\alpha_{\ell+1},\alpha_{\ell+2},\dots,\alpha_{\ell+n}$ are identical---\emph{i.e.}, that $\alpha_{k+i}=\alpha_{\ell+i}$ for all $1\leq i\leq n$.  We let $\aaa=\{A_{k,\ell,n}\st n>N_\ep, (\ell-k)<f(n)\}$ for some $N_\ep$ which depends only on $\ep$.  We want to apply the Local Lemma to conclude that, with positive probability, none of the events from $\aaa$ occur.

We define the graph $G$ by letting the out-neighborhood $\Gamma(A_{k_0,\ell_0,n_0})$ of the event $A_{k_0,\ell_0,n_0}$ consist of all events $A_{k,\ell,n}$ for which $[k+1,k+n]\cup [\ell+1,\ell+n]$ intersects $[\ell_0+1,\ell_0+n]$.  Finally, we define the variables 
\[
x_{k,\ell,n}=x_{A_{k,\ell,n}}=\frac 1 {f(n)n^3}.
\]

To check that condition (\ref{l.lll}) holds, first observe that $\pp(A_{k_0,\ell_0,n_0})\leq 2^{-(n_0-1)/2}$, since for any selection of the elements in the block $\alpha_{k_0+1},\dots,\alpha_{k_0+n_0}$, only one assignment of the at least $\frac {n_0-1}{2}$ variables $\eps_i$ ($\ell_0<i\leq \ell_0+n_0$, $i$ odd) allows the event $A_{k_0,\ell_0,n_0}$ (observe that this is the case even if $[k_0+1,k_0+n_0]$ and $[\ell_0+1,\ell_0+n_0]$ overlap).  Moreover, we have that $A_{k_0,\ell_0,n_0}$ is mutually independent of all the events in any family $\ccc\sbs \aaa\stm \Gamma(A).$ Thus, in particular, we have that
\begin{equation}
\pp\left(A_{k_0,\ell_0,n_0}|\bigcap_{C\in\ccc}\bar C\right)=\pp\left(A_{k_0,\ell_0,n_0}\right)\leq 2^{-(n_0-1)/{2}}.
\label{l.pbound}
\end{equation}
To show (\ref{l.lll}), it remains to check that, for any event $A_{k_0,\ell_0,n_0}\in \aaa$, we have 
\begin{equation}
x_{k_0,\ell_0,n_0}\prod_{\substack{A_{k,\ell,n}\in\\ \Gamma(A_{k_0,\ell_0,n_0})}} (1-x_{k,\ell,n})\geq 2^{-(n_0-1)/2}.
\label{l.xprodb}
\end{equation}
For any fixed $k_0,\ell_0,n_0$ and any fixed $n$, there are $<2(n+n_0)f(n)$ choices of $k,\ell,n$ satisfying $\ell-k<f(n)$ such that at least one of the intervals $[k+1,k+n],[\ell+1,\ell+n]$ overlaps with the interval $[\ell_0+1,\ell_0+n]$.

Thus, we can bound the product in line (\ref{l.xprodb}) as
\diffblock{\begin{multline}
x_{k_0,\ell_0,n_0}\prod_{\substack{A_{k,\ell,n}\in\\ \Gamma(A_{k_0,\ell_0,n_0})}} (1-x_{k,\ell,n})>
x_{k_0,\ell_0,n_0}\prod_{n=N_\ep}^\infty \left(1-\frac 1 {f(n)n^3}\right)^{2(n+n_0)f(n)}\\>
\frac{1}{f(n_0)n_0^3}\left(1-\sum_{n=N_\ep}^\infty \frac 1 {n^3}\right)^{2n_0}\left(1-\sum_{n=N_\ep}^\infty \frac 1 {n^2}\right)^2\\>
\frac{(2-\ep)^{-\frac {n_0} 2}}{n_0^3}\left(1-\frac 1 {N_\ep}\right)^{2n_0+2}
>2^{-(n_0-1)/ 2}
\label{l.ecalc}
\end{multline}}
for sufficiently large $N_\ep$ depending only on $\ep$ (and given that $n_0\geq N_\ep$).  Lines (\ref{l.pbound}) and (\ref{l.ecalc}) together give the condition (\ref{l.lll}), so, by the Local Lemma, there is a positive probability that none of the events $A_{k,\ell,n}$ ($n>N_\ep$, $\ell-k<f(n)$) occur, giving us Theorem \ref{t.feasy}.  Theorem \ref{t.easy} follows by a straightforward compactness argument.
\end{proof}

\vspace{1em}
What goes wrong when we try to apply the Local Lemma to prove Theorem \ref{t.2game}?  It is still true that $P(A_{k_0,\ell_0,n_0})\leq 2^{-(n_0-1)/2}$.  It is even still true that we have $P(A_{k_0,\ell_0,n_0}|\bigcap_{C\in\ccc}\bar C)\leq 2^{-(n_0-1)/2}$, so long as $\ccc$ only consists of events $A_{k,\ell,n}$ coming \emph{earlier} than $A_{k_0,\ell_0,n_0}$, in the sense that $\ell+n\leq \ell_0$, since no terms in the sequence produced by the game can affect the outcome of coin flips made by Player 1 to determine his moves later in the game.

The problem occurs when $\ccc$ contains events $A_{k,\ell,n}$ occurring \emph{after} $A_{k_0,\ell_0,n_0}$; say for example, $k>\ell_0+n_0$.  The problem with this situation is that  Player 2's strategy of play during these later intervals may well depend on whether or not the event $A_{k,\ell,n}$ occurred.  For example, if Player 2 is a very `sportsmanlike' player, he might `go easy' on Player 1 by helping Player 1 avoid as many bad interval pairs as possible as soon as he succeeds at winning by creating one such pair.  In this situation, we see that we cannot argue any useful upper bound on the probability $P(A_{k_0,\ell_0,n_0}|\bigcap_{C\in\ccc}\bar C)$, since the occurrence of the event $\bigcap_{C\in \ccc}\bar C$ would suggest that Player 2 is playing `sportsmanlike', and so has previously won, increasing the probability of the event $A_{k_0,\ell_0,n_0}$.

This problem seems quite annoying, since the problem only arises from the possibility that Player 2's behavior depends on whether or not he has previously secured a win by producing a bad interval, and yet, from the standpoint of whether or not a given strategy for Player 2 is a winning strategy, it is irrelevant how Player 2 plays after securing a win.  Morally speaking, it maybe seems enough that Player 1 can always do well no matter how the game has gone so far, and unreasonable to require him to play well regardless of how the game will continue.  In fact, this intuition is correct, which is the purpose of the `Lefthanded' version of the Local Lemma proved in the next section.

\section{Lefthanded Local Lemma}
\label{s.lllll}
The Local Lemma can be seen as generalizing the basic fact that if $\aaa$ is a family of independent events with probabilities $<1$, then $\pp(\bigcap_{A\in \aaa}\bar A)=\prod_{A\in \aaa}\pp(\bar A)>0$.  The `lopsided' form of the Local Lemma due to Erd\H{o}s and Spencer (Theorem \ref{t.lll}) can be seen as generalizing the fact that if $\pp(\bar A|\bigcap_{\ccc}\bar C)>0$ for all families $\ccc\sbs \aaa$, then 
\[
\pp\left(\bigcap_{A\in \aaa} \bar A\right)=\pp\left(\bar A'|\bigcap_{A\in \aaa\stm \{A'\}} \bar A\right)\pp\left(\bigcap_{A\in \aaa\stm \{A'\}}\bar A\right)>0,
\]
where `$>0$' follows by induction on $\abs{\aaa}$.

Of course, the condition $\pp(A|\bigcap_{\ccc}\bar C)<1$ for all families $\ccc\sbs \aaa$ is much more restrictive than is necessary for this kind of conclusion.  If we write $\aaa=\{A_1,A_2,\dots,A_m\}$, then it is sufficient, for example, to have that 
\[
(\forall i)(\forall \ccc\sbs \{A_j|j<i\})\;\;\;\pp\left(A_i|\bigcap_{\ccc}\bar C\right)<1,
\]
since in this case we can write $\pp\left(\bigcap_{A_j\in \aaa}A_j\right)$ as
\[
\pp\left(A_m|\bigcap_{j<m}A_j\right)\pp\left(A_{m-1}|\bigcap_{j<m-1}A_j\right)\cdots\pp\left(A_2|A_1\right)\pp\left(A_1\right)>0.
\]
The following `Lefthanded' version of the Local Lemma allows an analogous relaxation of the conditions of the Local Lemma. Recall that a \emph{quasi-order} is a transitive and reflexive (but not necessarily antisymmetric) binary relation.
\begin{theorem}
  Consider a family of events $\aaa$ as the vertices of a directed graph $G$, and endowed with some quasi-order $\leq$ such that
  \begin{equation}
(\forall A\in \aaa)(\forall B\in \Gamma(A),C\not\in \Gamma(A)\cup A) \;\; C\not>B \textrm{ or } C>A.
    \label{l.Dtop}
  \end{equation}
Assume further that there is an assignment of real numbers $0<x_A<1$ ($A\in \aaa$) such that for any $A\in \aaa$, and any family $\ccc\sbs \aaa\stm \Gamma(A)$ satisfying $C\not> A$ for all $C\in\ccc$, we have
  \begin{equation}
    \pp\left(A|\bigcap_{C\in \ccc} \bar C\right)\leq x_A\prod_{B\gets A}(1-x_B).
    \label{l.cond}
  \end{equation}
Then we have
  \begin{equation}
    \pp\left(\bigcap_{A_1\in \aaa_1}\bar A_1|\bigcap_{A_2\in \aaa_2}\bar A_2\right)\geq \prod_{A_1\in \aaa_1}(1-x_{A_1})
    \label{l.full}
  \end{equation}
  for any disjoint families $\aaa_1$, $\aaa_2$ for which $A_2\not> A_1$ for all $A_1\in \aaa_1$, $A_2\in \aaa_2$.  In particular,
  \begin{equation}
    \pp\left(\bigcap_{A\in \aaa}\bar A\right)>0.
  \end{equation}
  \label{t.ol}
\end{theorem}
\noindent This is just the regular Local Lemma in the case where $A\leq A'$ and $A'\leq A$ for all pairs $A,A'\in \aaa$, or alternatively when all pairs are incomparable.  It is more general since line (\ref{l.cond}) is only required to hold for families $\ccc\sbs \aaa\stm \Gamma(A)$ satisfying $C\not> A$ for all $C\in\ccc$, whereas standard versions of the Local Lemma (for example Theorem \ref{t.lll}) depend on this inequality for \emph{all} families $\ccc\sbs \aaa\stm \Gamma(A)$.

When is condition (\ref{l.Dtop}) satisfied?  One important case is when the dependency graph is a directed interval graph, in which the vertices correspond to intervals of $\N$; the quasi-order $\leq$ on intervals is the natural one induced by the relative positions of the right-endpoints of intervals: $[a,b]\leq [c,d]$ whenever $b\leq d$;  and $I_1\to I_2$ for intervals $I_1$ and $I_2$ if and only if $I_1$ and $I_2$ overlap and $I_2\leq I_1$.  (Note that in this case $\leq$ is in fact a total quasi-order, thus $[a,b]\not>[c,d]$ if and only if $[a,b]\leq [c,d]$.)  This is the case that we will use in all of our results on nonrepetitive sequences.  It seems that essentially the same special case was used by Peres and Schlag in their paper on Lacunary sequences\cite{PS}.

The proof of Theorem \ref{t.ol} is very similar to the standard proofs of the Local Lemma (\emph{e.g.,} those in \cite{pm} or \cite{vutao}.)  In fact, the following proof is essentially identical to these, except for the role of the quasi-order $\leq.$   We will abuse notation slightly by using $\aaa'\not>\aaa''$ to mean that $A'\not>A''$ for \emph{all} pairs $A'\in \aaa'$, $A''\in \aaa''$.

\begin{proof}
  Let $\aaa_1,\aaa_2$ be disjoint families of events from $\aaa$ satisfying $\aaa_2\not> \aaa_1$.  We prove line (\ref{l.full}) by induction on $m=\abs {\aaa_1}+\abs {\aaa_2}$.  We consider two cases.\\
\textbf{Case 1:} $\abs {\aaa_1}=1$.  For this case we need to show that
  \begin{equation}
    \pp\left(A|\bigcap_{A_2\in \aaa_2}\bar A_2\right)\leq x_{A}
    \label{l.case1}
  \end{equation}
  for $\aaa_1=\{A\}$.
  Write $\bbb=\aaa_2\cap \Gamma(A)$, and $\ccc=\aaa_2\stm \Gamma(A)$.
  
  From the identity $\pp(A|B\cap C)=\pp(A\cap B|C)/\pp(B|C)$, we have 
  \begin{equation}
    \pp\left(A|\bigcap_{A_2\in \aaa_2}\bar A_2\right)=
\frac
{ \pp\left(A\cap \bigcap_{B\in \bbb} \bar B| \bigcap_{C\in \ccc} \bar C\right)}
{\pp\left(\bigcap_{B\in \bbb}\bar B|\bigcap_{C\in \ccc}\bar C\right)}.
\label{l.frac}
  \end{equation}
We have that the numerator in (\ref{l.frac}) satisfies 
\begin{equation}
 \pp\left(A\cap \bigcap_{B\in \bbb} \bar B| \bigcap_{C\in \ccc} \bar C\right)\leq \pp\left(A | \bigcap_{C\in \ccc} \bar C\right)\leq x_A\prod_{B\gets A}(1-x_{B}),
\label{l.top}
\end{equation}
where the second inequality follows the condition (\ref{l.cond}).

By the condition (\ref{l.Dtop}) on the graph $G$, we have that $\ccc\not>\bbb$.  Thus, since $\abs {\bbb}+\abs{\ccc}=m-1$, we can apply line \ref{l.full} by induction to conclude that 
\begin{equation}
  \pp\left(\bigcap_{B\in \bbb}\bar B|\bigcap_{C\in \ccc}\bar C\right)\geq \prod_{B\in \bbb}(1-x_{B})\geq \prod_{B\gets A}(1-x_B).
\label{l.bot}
\end{equation}
Applying the bounds from (\ref{l.top}) and (\ref{l.bot}) to the identity in (\ref{l.frac}), we obtain the bound in line (\ref{l.case1}).

\vspace{1em}
\noindent \textbf{Case 2:} We reduce the case where $\abs{\aaa_1}\geq 2$ to the previously handled case of $\abs{\aaa_1}=1$ as follows:  let $\aaa_1=\{A\}\dot\cup \aaa_1'$, where $A$ is a maximal element of $\aaa_1$ under the order $\leq$.  By applying the identity $\pp(A\cap B| C)=\pp(A| B\cap C)\pp(B|C)$, we have 
\begin{equation}
P\left(\bigcap_{A_1\in \aaa_1}\bar A_1| \bigcap_{A_2\in \aaa_2} \bar A_2\right)=
P\left(\bar A|\bigcap_{A'\in \aaa_1'\cup \aaa_2} \bar A'\right)P\left(\bigcap_{A_1'\in \aaa_1'} \bar A_1'|\bigcap_{A_2\in \aaa_2}\bar A_2\right).
\label{l.case2}
\end{equation}
Notice here that, since we have $\aaa_2\not> \aaa_1$, we have as well that $\aaa_2\not> \aaa_1'$, and also $\aaa_1'\cup \aaa_2\not>\{A\}$.  Thus, applying line (\ref{l.case1}) from the previously handled case $\abs{\aaa_1}=1$ to bound the first term of the product in line (\ref{l.case2}), and using line (\ref{l.full}) by induction to bound the second, we get that
\begin{equation}
  P\left(\bigcap_{A_1\in \aaa_1}\bar A_1| \bigcap_{A_2\in \aaa_2} \bar A_2\right)\geq \prod_{A_1\in \aaa_1}(1-x_{A_1}).
\end{equation}
\end{proof}

Apart from the independent application by Peres and Schlag\cite{PS} of this kind of Local Lemma to Lacunary sequences, Grytczuk, Przyby\l{}o, and Zhu have used the Lefthanded Local Lemma, together with restricted sampling, to get near-optimum results for the Thue choice number\cite{Tchoice}.

\section{Thue-type binary sequence games}

\subsection{Long identical intervals can be made far apart}
\label{s.LII}
In this section we prove Theorem \ref{t.2game}.  Armed with the `Lefthanded' version of the Local Lemma, this will be no more difficult than was proving Theorem \ref{t.easy}. 
\begin{proof}
Again, we first consider a finite version where the game consists of just $M$ moves for some $M$.  We fix any strategy for Player 2 and let Player 1 play randomly against that strategy, flipping a coin to choose each of his moves independently.  We let the events $A_{k,\ell,n}\in \aaa$ be defined as in Section \ref{s.easy}, and define the total quasi-order on the events in $\aaa$ by letting $A_{k',\ell',n'}\leq A_{k,\ell,n}$ whenever $\ell'+n'\leq \ell+n$.  We define the graph $G$ now by letting $A_{k_0,\ell_0,n_0}\to A_{k,\ell,n}$ whenever $\ell+n\leq \ell_0+n_0$ and $[\ell+1,\ell+n]$ intersects $[\ell_0+1,\ell_0+n_0]$.  (Observe that this graph actually has significantly \emph{fewer} edges than the one used in the proof of Theorem \ref{t.easy}, since we only need to worry about overlaps `in one direction').

  With this setup, notice that for any set $\ccc\sbs \aaa\stm \Gamma(A_{k_0,\ell_0,n_0})$ such that $C\leq A_{k_0,\ell_0,n_0}$ for all $C\in \ccc$, the events in $\ccc$ do not affect the probabilities of coin flips made by Player 1 to choose his moves in the interval $[\ell_0+1,\ell_0+n]$; thus, we have 
\[
\pp\left(A_{k_0,\ell_0,n_0}|\bigcap_{C\in \ccc}\bar C\right)\leq 2^{-(n-1)/2}.
\]
Since the graph $G$ we define for this application is a proper subgraph of the graph used for the proof of Theorem \ref{t.easy}, the calculation in line (\ref{l.ecalc}) shows that the Lefthanded Local Lemma applies with the assignment $x_{k,\ell,n}=\frac{1}{f(n)n^3}$, where again $f(n)=(2-\ep)^{(n-1)/2}$.  (In fact, the smaller graph $G$ in this case allows the assignment  $x_{k,\ell,n}=\frac{1}{f(n)n^2}$.)  Thus the Lefthanded Local Lemma shows that, with positive probability, Player 1 defeats Player 2 \emph{regardless of the strategy chosen by Player 2}.   Since this implies that Player 2 has no winning strategy and the game is finite, Player 1 has a winning strategy.  A straightforward compactness argument implies that Player 1 has a winning strategy in the infinite version of the game.  Since it is unusual to prove the existence of a winning strategy by compactness, we give the whole argument.

Let $G_M$ denote the sequence game discussed above, played for $M$ moves.  Given first-player strategies $s$ for the game $G_M$ and $s'$ for the game $G_{M'}$, $M'>M$, we say that $s$ is an \emph{initial strategy} of $s'$ if the two strategies always agree during the first $M$ moves of any game.  

 Let $S_{M,\ep}$ be the set of all strategies $s$ for Player 1 in the game $G_M$ which are `winning' strategies for Player 1 for the fixed value $\ep$. In the case of Theorem \ref{t.2game}, we mean that for any strategy $s\in S_{M,\ep}$ there exists an $N_{\ep,s}$ so that in any sequence resulting from the game $G_M$ where Player 1 plays with the strategy $s$, any identical blocks of lengths $n>N_{\ep,s}$ are separated by at least the distance $(2-\ep)^{\frac n 2}$.  Our proof above for the finite version of Theorem \ref{t.2game} implies that for every $M$ and $\ep>0$, $S_{M,\ep}$ is nonempty.  

Fix some $\ep>0$ and consider $\bigcup_{M=0}^\infty S_{M,\ep}$ as the vertices of a tree where an element $s\in S_{M,\ep}$ is joined to an element $s'\in S_{M+1,\ep}$ whenever $s$ is an initial strategy of $s'$.  This tree has finite degree, and it has infinitely many vertices, by our proof above all the $S_{M,\ep}$'s are nonempty.  Thus K\"onig's Infinity Lemma implies that there is a sequence of strategies $s_1,s_2,s_3,\dots$ with $s_i\in S_{i,\ep}$ and such that $s_i$ is always an initial strategy of $s_{j}$ whenever $i<j$.  To play with a winning strategy in the infinite sequence game $G_\infty$, Player 1 makes his 1st move according to the strategy $s_1$, his second move (the third move of the game) with the strategy $s_3$, and in general, makes his $k$th move, the $(2k-1)$st move of the game, according to strategy $s_{2k-1}$.  The fact that $s_i$ is an initial strategy of $s_j$ for $i<j$ implies that every move he makes is made consistent with all strategies he will ever play with.  This implies that no bad intervals can show up in the first $M$ moves of play for any $M$.  Thus no `bad pairs' of intervals can show up at all, and Theorem \ref{t.2game} is proved.
\end{proof}

The proof of Theorem \ref{t.tgame} is very similar to that of Theorem \ref{t.2game}, and we omit it.\qed

\subsection{Adjacent intervals can be made very different}
In this section we prove Theorem \ref{t.agame}.  The proof is is hardly changed from the proof of the Exercise in \cite{pm} which motivates it, apart from making use of the Lefthanded version of the Local Lemma.  
\begin{proof}
Again we begin by restricting to a finite number of moves $M$.  We construct a family of events $\aaa=\{A_{k,n}\}$ ($n>N_{\ep,t}$, $0\leq k\leq M-n$) by letting  $A_{k,n}$ denote the event that the adjacent blocks $\alpha_{k+1,},\dots,\alpha_{k+n}$ and $\alpha_{k+n+1},\dots,\alpha_{k+2n}$ agree in $\geq (1-\frac{1}{2t+2}+\ep)n$ places.  We define a total quasi-order $\leq$ on $\aaa$ by letting $A_{k,n}\leq A_{k',n'}$ whenever $k+2n\leq k'+2n'$, and define the dependency digraph $G$ by letting $A_{k_0,n_0}\to A_{k,n}$ whenever $A_{k,n}\leq A_{k_0,n_0}$ and $[k+1,k+2n]$ intersects $[k_0+n_0+1,k_0+2n_0]$.  Since a family $\ccc\sbs \aaa\stm \Gamma(A_{k_0,n_0})$, $\ccc\leq A_{k_0,n_0}$ consists of events which are independent of all the at least $\lfloor\frac {n_0}{t+1}\rfloor$ coin flips used by Player 1 to determine his moves in the block $\alpha_{k_0+n_0+1},\dots,\alpha_{k_0+2n_0}$, we have that 
\diffblock{\begin{multline}
\label{l.acs}
\pp\left(A_{k_0,n_0}|\bigcap_{C\in \ccc}\bar C\right)=
P\left(A_{k_0,n_0}\right)\leq 
\frac{1}{2^{\lfloor {n_0}/({t+1})\rfloor}}\sum_{j=\clg{(\frac 1 {2t+2}+\ep){n_0}}}^{\lfloor n_0/(t+1)\rfloor}\binom{\lfloor\frac {n_0}{t+1}\rfloor}{j}\\\leq 
\frac{ {n_0}/(2t+2)}{2^{\lfloor {n_0}/({t+1})\rfloor}} \binom{\lfloor\frac {n_0}{t+1}\rfloor}{\clg{(\frac 1 {2t+2}+\ep){n_0}}}<
\frac {n_0} {2^{{n_0}/(t+1)}}\binom{\lfloor\frac {n_0}{t+1}\rfloor}{\clg{(\frac 1 {2t+2}+\ep){n_0}}}.
\\
\end{multline}}
For any $\ep_0>0$, there exists a $\alpha<1$ (\emph{e.g.}, $\alpha=1/\sqrt{1+4\ep_0^2}$) such that 
\begin{equation}
 \binom{N}{\clg{(\frac 1 2+\ep_0)N}}\leq (\alpha 2)^N
\end{equation}
  Letting $N=\frac{n_0}{t+1}$ and $\ep_0=(t+1)\ep$, we get (from line (\ref{l.acs})) that 
\begin{equation}
\pp\left(A_{k_0,n_0}|\bigcap_{C\in \ccc}\bar C\right)<
\frac {n_0} {2^{{n_0}/(t+1)}}\binom{{n_0}/(t+1)}{\clg{(\frac 1 {2t+2}+\ep){n_0}}}\leq
n_0 \alpha^{n_0/(t+1)}
\label{l.adjpbound}
\end{equation}
for a constant $\alpha<1$.
\noindent On the other hand, letting $x_{A_{k,n}}=x_{k,n}=b^{n}$ ($0<b<1$ will be specified later), we have that 
\diffblock{\begin{multline}  
  x_{A_{k_0,n_0}}\hspace{-3.5ex}\prod_{A_{k,n}\gets A_{k_0,n_0}}\hspace{-3ex}(1-x_{A_{k,n}})>
x_{k_0,n_0}\hspace{-1ex}\prod_{n=N_{\ep,t}}^\infty \prod_{k=k_0-n}^{k_0+n_0-n}\hspace{-1ex}(1-x_{k,n})\\=
b^{n_0}\hspace{-1ex}\prod_{n=N_{\ep,t}}^\infty \left(1-b^{n}\right)^{n_0}\geq
b^{n_0}\left(1-\sum_{n=N_{\ep,t}}^\infty b^{n}\right)^{n_0}
\label{l.bprodin}
\end{multline}}
And now choosing $b$ between $\alpha^{1/(t+1)}$ and 1 and letting $N_{\ep,t}$ be sufficiently large, lines (\ref{l.adjpbound}) and (\ref{l.bprodin}) give us that
\begin{equation}
 x_{A_{k_0,n_0}}\hspace{-3.5ex}\prod_{A_{k,n}\gets A_{k_0,n_0}}\hspace{-3ex}(1-x_{A_{k,n}})>
P\left(A_{k_0,n_0}|\bigcap_{C\in \ccc} \bar C\right),
\end{equation}
for any $\ccc\sbs \aaa\stm \Gamma(A)$ such that $\ccc\leq A$.  Thus the Lefthanded Local Lemma applies, and we have that $P\left(\bigcap_\aaa \bar A_{k,n}\right)>0$; thus regardless of the choice of strategy for Player 2, it cannot be a winning strategy and thus Player 1 has a winning strategy.   Compactness implies he has a winning strategy in the infinite version of the game.
\end{proof}

We point out here that it is actually possible to have a theorem which combines Theorems \ref{t.tgame} and \ref{t.agame}; indeed, both for the original game-free results which motivated them, and these theorems, examining the proofs shows that both types of `bad events' can be avoided simultaneously.

\section{$c$-ary nonrepetitive sequence games}
\label{s.rainbow}
Beck's Theorem \ref{t.beck} and the Alon-Spencer exercise both imply the existence of strictly nonrepetitive sequences (no consecutive identical blocks of \emph{any} lengths) of sufficiently large base.  Their game-theoretic analogs (Theorems \ref{t.2game} and \ref{t.agame}), however, do not imply that Player 1 can force the production of a nonrepetitive sequence in the $c$-ary sequence game for any $c$---and, indeed, Player 2 can certainly produce lots of identical adjacent pairs of blocks of length 1 just by mimicking Player 1's moves.  Something along the lines suggested here does hold, however.  In fact, we can prove a game-theoretic analog of the following theorem of Grytczuk\cite{rainbow}, whose proof uses the Local Lemma:
\begin{theorem}[Grytczuk]
  Let $k\geq 2$ be a fixed integer.  There is a $[c]$-coloring $\chi$ of $\N$, $c\leq \frac 1 2 e^{k(4k-2)/(k-1)^2}k^2(k-1)$, such that for every $r\geq 1$, every block of length $kr$ contains a $k$-term rainbow arithmetic progression of difference $r$.  In particular, among any $k$ consecutive blocks of the same length in the sequence $\chi(1)\chi(2)\chi(3)\dots$, no two are identical.\qed
\label{t.gry}
\end{theorem}
Here a $[c]$-coloring is an assignment $\N\to \{1,2,\dots,c\}$, and a \emph{rainbow arithmetic progression} is an arithmetic progression all of whose terms get different colors.  Note that the conclusion in Grytczuk's theorem regarding consecutive blocks now follows (with best possible $c=k+1$) from the recent proofs \cite{crD,rD} of Dejean's conjecture.

 For our game-theoretic version, we cannot expect Player 1 to always be able to force the construction of rainbow arithmetic progressions, since for all odd $r$, essentially half of the terms of any arithmetic progression of difference $r$ are controlled by Player 2 (and so may all be the same color, for example).  For odd $r$, Player 1 will instead create \emph{prismatic pairs} of arithmetic progressions.  A prismatic pair of $k$-term  arithmetic progressions $\alpha_1,\alpha_2,\dots,\alpha_k$ and $\beta_1,\beta_2,\dots,\beta_k$  (with respect to a $c$-coloring $\chi$) is a pair for which we have $\beta_i=\alpha_i+1$ for all $1\leq i\leq k$, and $\chi(\alpha_i)\neq \chi(\alpha_j)$ for all $\alpha_i<\alpha_j$, $\alpha_j$ odd, and similarly, $\chi(\beta_i)\neq \chi(\beta_j)$ for all $\beta_i<\beta_j$, $\beta_j$ odd. 

The important thing about prismatic pairs of arithmetic progressions is that they are essentially as useful as rainbow arithmetic progressions from the standpoint of consecutive blocks:
\begin{observ}
  Under any coloring of the natural numbers, any interval $I$ of length $kr$, $(r\geq 2)$ containing a prismatic pair of $k$-term arithmetic progressions of difference $r$ has the property that no two of the $k$ consecutive intervals of length $r$ which make up $I$ are identical.\qed
\end{observ}
\noindent We are now ready for our game-theoretic version of Grytczuk's theorem.
\begin{theorem}
  For any fixed $k\geq 2$, there is some $C_k$ (\emph{e.g,} $C_2=37$, $C_k\lesssim 3ek^3$), such that for any integer $c\geq C_k$, Player 1 has a strategy in the $c$-ary sequence game  which ensures that for every $r\geq 2$, every block of length $\geq kr$ contains either a $k$-term rainbow arithmetic progression of difference $r$ (if $r$ is even) or a prismatic pair of $k$-term arithmetic progressions of difference $r$ (if $r$ is odd).  In particular, among any $k$ consecutive blocks of any equal length $r\geq 2$ in the sequence resulting from gameplay, no two are identical.
\label{t.grain}
\end{theorem}
Observe that the smallest case $k=2$ of Theorem \ref{t.grain} is a game-theoretic analog to Thue's original theorem on nonrepetitive sequences.

The proof of Theorem \ref{t.grain} is very similar to that of Grytczuk's theorem except we must apply the Lefthanded version of the Local Lemma, and we must be content to find prismatic pairs of arithmetic progressions when we are not guaranteed to find a rainbow arithmetic progression.

\begin{proof}  As usual, we begin by fixing some finite $M$ and will first prove that Player 1 has a suitable strategy in the finite game, played for just $M$ moves.   Fixing any deterministic strategy $\sigma_2$ for Player 2, Player 1 chooses each of his moves from $[c]$ randomly and independently (on each turn, any choice has probability $\frac 1 c$).  For even $r\geq 2$, we let the event $A_{\ell,r}$ denote the event that, after $M$ moves of play, the interval $I_{\ell,r}=[\ell+1,\ell+kr]$ contains no $k$-term arithmetic progression of difference $r$, and for odd $r>2$, we let the event $A_{\ell,r}$ denote the event that the interval $I_{\ell,r}$ contains no prismatic pair of $k$-term arithmetic progressions of difference $r$.  We define a total quasi order $\leq$ on the set $\aaa$ of events $A_{\ell,r}$ by letting $A_{\ell,r}\leq A_{\ell',r'}$ whenever $\ell+kr\leq \ell'+kr'$, and define the dependency graph by letting $A_{\ell',r'}\to A{\ell,r}$ whenever $A_{\ell,r}\leq A_{\ell',r'}$ and the intervals $I_{\ell,r}$ and $I_{\ell',r'}$ overlap.   Fix some event $A_{\ell_0,r_0}$, and let $\ccc$ be a family of events $A_{\ell,r}\leq A_{\ell_0,r_0}$ nonadjacent to $A_{\ell_0,r_0}$. 

We claim we have
  \begin{equation}
    P\left(A_{\ell_0,r_0}|\bigcap_{C\in \ccc}\bar C\right)\leq \frac 1 {c^{\flr{\frac {r_0} 2}}}\binom{k}{2}^{\flr{\frac {r_0} 2}}.
\label{l.PAupp}
  \end{equation}
Observe that $I_{\ell_0,r_0}$ contains $r_0$ $k$-term arithmetic progressions of difference $r_0$.  In the case where $r_0$ is even, $\flr{\frac {r_0} 2}=\frac {r_0} 2$ of these consist entirely of elements whose colors are chosen by Player 1.  For none of these to be rainbow progressions, each must have a pair of elements which get the same color.  Since Player 1 makes his choices of colors independently of previous moves made in the game, we have for a fixed pair of elements that the probability is $\frac 1 c$ that they get the same color (even conditioning on the event $\bigcap_{C\in \ccc}\bar C$), and there are $\binom{k}{2}$ such pairs for each progression, giving the upper bound in line (\ref{l.PAupp}) when $r_0$ is even.

For the case where $r_0$ is odd, observe that we can group the $r_0$ $k$-term arithmetic progressions of difference $r_0$ in $I_{\ell_0,r_0}$ into at least $\flr{\frac {r_0} 2}$ consecutive pairs.  A pair of consecutive $k$-term arithmetic progressions $\alpha_1,\dots,\alpha_k$ and $\beta_1,\dots \beta_k$, $\beta_i=\alpha_i+1$ is prismatic unless we have either that $\chi(\alpha_i)= \chi(\alpha_j)$ for some $\alpha_i<\alpha_j$ and $\alpha_j$ odd, or that $\chi(\beta_i)= \chi(\beta_j)$ for some $\beta_i<\beta_j$ and $\beta_j$ odd.  There are $\binom{k}{2}$ possible such pairs, since each pair $(i,j)$ ($1\leq i<j\leq k$) corresponds to exactly one of these pair-types.  Since Player 1 chooses the color of odd $\alpha_i$'s and $\beta_i$'s independently of all previous moves in the game, each pair has probability $\frac 1 c$ of being monochromatic.  Thus we have the upper bound in line (\ref{l.PAupp}) when $r_0$ is odd as well.

For any fixed interval $I_{\ell_0,r_0}$ and any fixed $r$, observe that there are at most $\abs{I_{\ell_0,r_0}}=kr_0$ intervals $I_{\ell,r}$ which intersect $I_{\ell_0,r_0}$ and come before it ($\ell+kr\leq \ell_0+kr_0$).  We set $x_{A_{\ell,r}}=x_{r}=a_k^{\flr{r/2}}$, where $a_k\leq \frac 1 k$ is a constant (depending on $k$) to be specified later.  (From an asymptotic point of view, $a_k=\frac 1 k$ is essentially the best choice, but we are especially interested in the case $k=2$.)  We have
\diffblock{\begin{multline}
  x_{A_{\ell_0,r_0}}\hspace{-2ex}\prod_{B\gets A_{\ell_0,r_0}}\hspace{-2ex}(1-x_B)\geq 
a_k^{\flr{r_0/2}}\prod_{r=2}^{\infty}\left(1-a_k^{\flr{r/2}}\right)^{kr_0}\hspace{-1ex}=
a_k^{\flr{r_0/2}} \prod_{j=1}^\infty \left(1-a_k^j\right)^{2kr_0}
.
\label{l.rainxbound}
\end{multline}}
Let now 
\begin{equation}
C_k= \binom{k}{2} \phi(a_k)^{-6k} a_k^{-1},
\label{l.ck}
\end{equation}
where $\phi(a_k)=\prod_{j=1}^\infty \left(1-a_k^j\right)$ is Euler's q-series for $q=a_k$.  (Note that $\phi(\frac 1 k)^{-k}\to e$, so $\phi(a_k)^{-6k}$ is essentially playing the role of a constant in this expression.) 
For any integer $c\geq C_k$, we have
\begin{equation}
 P\left(A_{\ell_0,r_0}|\bigcap_{C\in \ccc}\bar C\right)\leq 
\frac 1 {{c}^{\flr{\frac {r_0} 2}}}\binom{k}{2}^{\flr{\frac {r_0} 2}}\leq 
\left(a_k\phi(a_k)^{6k}\right)^{\flr{\frac {r_0} 2}}\\\leq
{a_k^{\flr{r_0/2}}} \phi(a_k)^{2kr_0}
\label{l.rainpbound}
\end{equation}
where for the last inequality we are using the fact that $\flr{\frac {r_0} 2}/r_0\geq \frac 1 3$ for all $r_0\geq 2$.  Combining lines (\ref{l.rainpbound}) and (\ref{l.rainxbound}) we get that
\begin{equation}
   P\left(A_{\ell_0,r_0}|\bigcap_{C\in \ccc}\bar C\right)\leq x_{A_{\ell_0,r_0}}\hspace{-2ex}\prod_{B\gets A_{\ell_0,r_0}}(1-x_B),
\end{equation}
and so the Lefthanded Local Lemma applies, as desired, for $c$ sufficiently large as indicated.  We conclude that Player 2 has no strategy in any finite $c$-ary sequence game to ensure $k$-repetition of blocks of length $\geq 2$ ($c$ depends on $k$ here) since the Local Lemma implies that Player 1 may win just by random play.  Thus Player 1 has a strategy in the $c$-ary sequence game to \emph{prevent} any $k$-repetitions from occurring, thus, by compactness, Player 1 has such a strategy in the infinite version of the game.

Let's examine the requirement ${C_k}\geq \binom{k}{2} \phi(a_k)^{-6k} a_k^{-1}$.  By Euler's Pentagonal Number theorem, 
\[
\phi(a_k)=\sum_{r=-\infty}^\infty (-1)^r a_k^{r(3r-1)/2}=1- a_k - a_k^2 + a_k^5+\cdots> 1-a_k - a_k^2
\]
(the last inequality holding since $a_k\leq \frac 1 k\leq \frac 1 2$). Thus, in particular, requiring 
\begin{equation}
c\geq \frac 1 2 \left(1-a_k-a_k^2\right)^{-6k}a_k^{-1}k(k-1)>C_k
\label{l.bck}
\end{equation}
 suffices to ensure a winning strategy for Player 1.  This bound is minimized for each $k$ by letting 
\begin{equation}
a_k=\frac{\sqrt{36k^2+60k+5}-(6k+1)}{24k+2}.
\label{l.ak}
\end{equation}
 In particular, for $k=2$ and $a_2=.068\dots$, we get that $C_2\leq 37$, thus Player 1 has a strategy in the infinite 37-ary sequence game to ensure that there will be no consecutive identical blocks of lengths $\geq 2$.  Asymptotically, (\ref{l.ak}) gives $a_k\sim \frac 1 {6k}$, and the substitution $a_k=\frac 1 {6k}$ in line (\ref{l.ck}) gives that $C_k\lesssim 3ek^3$. 
\end{proof}

As with our other results, it is possible to prove a biased version of Theorem \ref{t.grain}.  There is another natural direction in which to go from Theorem \ref{t.grain}, however.  The case $k=2$ of Theorem \ref{t.grain} shows that there is a sufficiently large base ($\leq 37$, in fact) for which Player 1 has a strategy in the $c$-ary sequence game to avoid the production of any consecutive identical blocks of lengths $r\geq 2$.  What about repetition of blocks of length $r=1$?  Obviously Player 2 can force the production of two consecutive identical digits, but what about three in a row?  The following theorem, stated for a game of any bias, shows this is not the case:

\begin{theorem}
  For any fixed $t\geq 1$, there is some sufficiently large integer $C_t$ (\emph{e.g.,} $C_1\leq 64$, $C_t\lesssim \frac 9 2 et^3$) such that, for any integer $c\geq C_t$, Player 1 has a strategy in the $(1:t)$ $c$-ary sequence game to ensure that there are no consecutive identical blocks $\beta_1\beta_2\dots \beta_k$ ($k\geq 2$) with total length $\sum \abs{\beta_i}\geq 2t+1$. 
\label{t.sblocks}
\end{theorem}

On the other hand, it is not hard to see that for \emph{any} value of $c$, Player 2 can force the existence of a pair of consecutive identical blocks of length $t$ in the $(1:t)$ $c$-ary sequence game.  Thus Theorem \ref{t.sblocks} shows the sharp length threshold, for which Player 2 can force repetitions of any (strictly) smaller total length, and Player 1 can avoid any repetitions of any larger total length.

For the proof of Theorem \ref{t.sblocks}, we will again apply the ordered Local Lemma to a game of any finite length $M$, and infer the infinite version by compactness.

\begin{proof}
Let $I_{\ell,n}$ ($\ell+n\leq M$) denote the interval $[\ell+1,\ell+n]$.  As usual, we fix any strategy for Player 2, and let Player 1 choose his moves randomly from the set $[c]$.  We define two types of events.  Let $B_{\ell,n}$, $2t+1\leq n\leq 3t$ indicate the event that in the resulting sequence, the block corresponding to the interval $I_{\ell,n}$ is equal the concatenation of (at least two) consecutive identical blocks, and let $A_{\ell,r}$, $r\geq \clg{(3t+1)/2}$ denote the event that the interval $I_{\ell,2r}$ is the concatenation of exactly two identical blocks (each of length $r$).   Observe now that the event 
\[
\bigcap_{\substack{\ell\geq 1\\ 2t+1\leq n\leq 3t}} \bar B_{\ell,n} \, \cap \bigcap_{\substack{\ell\geq 1\\ r\geq (3t+1)/2}} \bar A_{\ell,r}
\]
implies that Player 1 has `won the game': there can be no consecutively repeated identical blocks of total length $\geq 2t+1$.

Note that $n\geq 2t+1$ implies that $\ell+j\equiv 1\pmod {t+1}$ for at least one $j$ in $\flr{\frac n 2}< j\leq n$.  As a consequence, at least one of the terms in the `second-half' subinterval $[\ell+\flr{\frac n 2}+1,\ell+n]$ is controlled by Player 1.  If on this turn he chooses an element from $[c]$ different from all moves made in the `first-half' subinterval $[\ell+1,\ell+\flr{\frac n 2}]$, then the event $B_{\ell,n}$ cannot occur.  Thus we have
\[
P(B_{\ell_0,n_0})\leq 1-\frac{c-\flr{\frac {n_0} 2}}{c}\leq \frac {n_0} {2c}\leq \frac {3t}{2c}.
\]
We have a much better bound on the probabilities of the events $A_{\ell,r}$, however:
\[
P(A_{\ell,r})\leq \frac{1}{c^{\flr{r_0/(t+1)}}}.
\]
Both types of events $B_{\ell,n}$ and $A_{\ell,r}$ have natural correspondences with intervals of natural numbers ($[\ell+1,\ell+n]$ and $[\ell+1,\ell+2r]$, respectively).  As usual, we let the quasi-order $\leq$ be induced by the right endpoints of these corresponding intervals, and define adjacency by letting $E_1\to E_2$ whenever $E_2\leq E_1$ and the intervals corresponding to $E_2$ and $E_1$ overlap.  We set $x_{B_{\ell,n}}=b_t$ and set $x_{A_{\ell,r}}=a_t^{\flr{r/(t+1)}}$, where $b_t,a_t\leq \frac 1 {t^2}$ are to be specified later. We have that

\diffblock{\begin{multline}
  x_{A_{\ell_0,r_0}}\prod_{B\gets A_{\ell_0,r_0}}(1-x_B)\\\geq 
a_t^{\flr{r_0/(t+1)}} \prod_{n=2t+1}^{3t}\left(1-b_t\right)^{2r_0}\prod_{r=\clg{(3t+1)/2}}^{\infty}\left(1-a_t^{\flr{r/(t+1)}}\right)^{2r_0}\\=
a_t^{\flr{r_0/(t+1)}} \left((1-b_t) \prod_{j=1}^\infty \left(1-a_t^j\right)\right)^{2(t+1)r_0}\\=
a_t^{\flr{r_0/(t+1)}}((1-b_t)\phi(a_t))^{2(t+1)}.
\end{multline}}
Thus, for the Lefthanded Local Lemma to apply, we must have
\begin{equation}
  \frac 1 {c^{\flr{r_0/(t+1)}}}\leq a_t^{\flr{r_0/(t+1)}}((1-b_t)\phi(a_t))^{2(t+1)},
\end{equation}
which holds so long as $c \geq a_t^{-1}((1-b)\phi(a_t))^{-(2t+1)(2t+2)}$ (since $r_0/\flr{r_0/(t+1)}\leq 2t+1$.)
Similarly, we have 
\diffblock{\begin{multline}
  x_{B_{\ell_0,n_0}}\prod_{B\gets B_{\ell_0,r_0}}(1-x_B)\geq 
b_t\prod_{n=2t+1}^{3t}\left(1-b_t\right)^{n_0}\prod_{r=\clg{(3t+1)/2}}^{\infty}\left(1-a_t^{\flr{r/(t+1)}}\right)^{n_0}\\=
b_t \left((1-b_t) \prod_{j=1}^\infty \left(1-a_t^j\right)\right)^{(t+1)n_0}\geq 
b_t\left((1-b_t)\phi(a_t)\right)^{3t^2+3}.
\end{multline}}
Thus for our application of the Local Lemma, we also require 
\begin{equation}
  \frac {3t}{2c}\leq b_t\left((1-b_t)\phi(a_t)\right)^{3t^2+3},
\end{equation}
which holds so long as 
$
  c\geq \frac 3 2 t b_t^{-1}\left((1-b_t)\phi(a_t)\right)^{-3t^2-3}.
$
Therefore, setting
\begin{equation}
  C_t=\max
\left\{
\begin{array}{l}
a_t^{-1}\left((1-b_t)(1-a_t-a_t^2)\right)^{-(2t+1)(2t+2)}\\
\frac 3 2 t b_t^{-1}\left((1-b_t)(1-a_t-a_t^2)\right)^{-3t^2-3},
\end{array}
\right.
\end{equation}
we have that the one-sided Local Lemma applies so long as $c\geq C_t$ and the theorem follows.

For the case $t=1$, we make the assignment $a_1=.0514$, $b_1=.0426$ (obtained by numerical optimization), which gives that the theorem holds with $C_1=64$.  For large $t$, the assignment $a_t=\frac 1 {t^{5/2}},b_t=\frac 1 {3t^2}$ gives that $C_t\lesssim \frac 9 2 et^3$.
\end{proof}

\section{Pattern avoidance}
\label{s.patt}
As discussed earlier, it is not known whether there might be some upper bound on the index of avoidable patterns.  Nevertheless, as discussed earlier, there \emph{are} finite bounds on the index of patterns without ioslated variables\cite{patgraph,bgoh}.  In this section we prove the following game-theoretic analog of those bounds.
\begin{theoremR}{t.gpattern}
For any pattern $p$ with no isolated variables, Player 1 has a strategy in the 429-ary sequence game to ensure that the sequence that results from game-play does not contain a word which matches $p$ under a partition consisting of blocks $B_i$ all of lengths $\abs {B_i}\geq 2$.
\end{theoremR}
Note that one could relax the pattern-matching condition in Theorem \ref{t.gpattern} and still get a finite bound; for example, it is sufficient if at least some constant fraction of the blocks $B_i$ which form the pattern match have lengths $\geq 2$.

For the proof, we are again applying the Lefthanded Local Lemma.  Let $p$ be a pattern $p_1p_2\dots p_k$.  We fix any strategy for Player 2 in the finite $c$-ary sequence game of length $M$, and let Player 1 play randomly against it.  As before, it is enough to show that, for sufficiently large $c$, Player 1 can win for every $M$.

We let the event $A_{\ell,n}$ ($\ell+n\leq M$) denote the event that the interval $I_{\ell,n}$ matches the pattern $p$ in the sense of Theorem \ref{t.gpattern}, and define a (total) quasi-order $\leq$ on the set $\aaa$ of events $A_{\ell,n}$ by letting $A_{\ell,n}\leq A_{\ell_0,n_0}$ whenever $\ell+n\leq \ell_0+n_0$.
Fixing the event $A_{\ell_0,n_0}$ and letting $C$ be any family of events $A_{\ell,n}\leq A_{\ell_0,n_0}$ whose corresponding intervals $I_{\ell,n}$ are disjoint from $I_{\ell_0,n_0}$, we claim that we have

\begin{equation}
P\left(A_{\ell_0,n_0}|\bigcap_{C\in \ccc}\bar C\right)
\leq \frac{2^{\flr{n_0/2}}}{c^{\flr{\clg{n_0/2}/2}}}
\leq \frac{4^{\flr{\clg{n_0/2}/2}}}{c^{\flr{\clg{n_0/2}/2}}}
\label{l.gppbound}  
\end{equation}

To get this bound, fix any partition of the interval $I_{\ell_0,n_0}$  into consecutive blocks $B_1,B_2,\dots,B_k$ each of length $\geq 2$ such that for all $1\leq i,j\leq k$, we have that $p_i=p_j$ implies that $\abs{B_i}=\abs{B_j}$.

 Fixing such a partition, the probability that the interval $I_{\ell_0,n_0}$ matches $p$ \emph{according to this partition} (\emph{i.e.}, that $p_i=p_j$ implies that $B_i=B_j$) is at most $\left(\frac 1 {c}\right)^{\flr{\clg{n_0/2}/2}}$, since any symbol occurring in $p$ occurs at least twice.  Thus we will get the bound in line (\ref{l.gppbound}) by showing that there are at most $2^{\flr{n_0/2}}$ ways of partitioning the interval $I_{\ell_0,n_0}$ into blocks whose pattern of lengths is consistent with the pattern $p$.  To see this, note that if we define an equivalence relation $\sim$ on blocks by $B_i\sim B_j$ whenever $p_i=p_j$, then a partitioning of the interval $I_{\ell_0,n_0}$ consistent with the pattern $p$ is determined by the choice of the lengths of the blocks in each equivalence class.  These lengths must sum to at most $\flr{n_0/2}$ (since each class contains at least two blocks), giving us the bound $2^{\flr{n_0/2}}$ on the number of possible partitions of $I_{\ell_0,n_0}$.  

\bigskip

Fixing $I_{\ell_0,n_0}$ and $n$, there are at most $n_0$ intervals $I_{\ell,n}\leq I_{\ell_0,n_0}$ which intersect $I_{\ell_0,n_0}$.  
Set $x_{A_{\ell,n}}=x_n=b^{\flr{\clg{\frac n 2}/2}}$, where $b\leq \frac 1 2$ is a constant to be optimized ($b=\frac 1 2$ would sufficient for the argument to give a finite bound).  We have
\diffblock{\begin{multline}
  x_{A_{\ell_0,r_0}}\hspace{-2ex}\prod_{B\gets A_{\ell_0,r_0}}\hspace{-2ex}(1-x_B)\geq 
b^{\flr{\clg{n_0/2}/2}}\prod_{n=4}^{\infty}\left(1-b^{\flr{\clg{n/2}/2}}\right)^{n_0}\hspace{-1ex}\\=
b^{\flr{\clg{n_0/2}/2}}\frac{1}{(1-b)^{n_0}}\prod_{j=1}^{\infty}\left(1-b^{j}\right)^{4n_0}\hspace{-1ex}
\label{l.gparbound}
\end{multline}}
Let now
\begin{equation}
C=\frac{8(1-b)^6}{b\phi(b)^{24}},
\label{l.gparC}
\end{equation}
where, as before, $\phi(a)=\prod_{j=1}^\infty \left(1-a^j\right)$ is Euler's q-series for $q=a$.  
For any integer $c\geq C$, this gives
\diffblock{\begin{multline}
 P\left(A_{\ell_0,r_0}|\bigcap_{C\in \ccc}\bar C\right)\leq 
\frac{2^{\flr{n_0/2}}}{c^{\flr{\clg{n_0/2}/2}}}\leq
\left(\frac{8}{c}\right)^{\flr{\clg{n_0/2}/2}}\hspace{-1ex}\\\leq
\left(\frac{b\phi(b)^{24}}{(1-b)^6} \right)^{\flr{\clg{n_0/2}/2}} \leq
b^{\flr{\clg{n_0/2}/2}}\frac{\phi(b)^{4n_0}}{(1-b)^{n_0}}.\\
\label{l.parxbound}
\end{multline}}
Combining this with (\ref{l.gparbound}), we get that 
\begin{equation}
   P\left(A_{\ell_0,r_0}|\bigcap_{C\in \ccc}\bar C\right)\leq 
  x_{A_{\ell_0,r_0}}\hspace{-2ex}\prod_{B\gets A_{\ell_0,r_0}}\hspace{-2ex}(1-x_B)
\end{equation}
and so the Local Lemma applies.
Under optimization of $b$, we get the theorem with $C=429$ (with $b=.045$, for example).

\section{Further Questions}
\label{s.qs}
There are many natural questions raised by what we have done here.  For example, regarding Theorem \ref{t.grain}:
\begin{q}
  What is the minimum $c$ for which Player 1 has a strategy in the $c$-ary sequence game to ensure that there are no consecutive identical blocks of any length $r\geq 2$ in the resulting $c$-ary sequence?
\label{q.cc}
\end{q}
We have proved an upper bound of 37. Grytczuk, Przyby\l{}o, and Zhu have used restricted sampling techniques to push the Lefthanded Local Lemma to give near optimal results for the Thue choice number  \cite{Tchoice} (the list-chromatic analog to nonrepetitive colorings), and it seems likely that some of their method could be used to decrease our upper bound of 37.  On the other hand, unlike in the problem they consider, it seems unlikely that this alone could get close to closing the gap between the bounds in our case.

Apart from decreasing the upper bound, there is the problem of lower bounds. If Player 2's strategy is to always take the digit 0, then any square in the sequence of just Player 1's moves will cause him to lose; this implies a lower bound of 3.  In fact, there is a simple strategy for Player 2 which shows that 3 is not the truth either:
\begin{theorem}
  Player 2 has a strategy in the 3-ary sequence game to ensure the production of consecutive identical blocks of length $\geq 2$.
\label{t.cl}
\end{theorem}
\begin{proof}
  Player 2's strategy is very simple: after any move by Player 1 on which he chooses the symbol $a\in \{0,1,2\}$, Player 2 chooses $a+1\pmod 3$.  Fixing this strategy for Player 2, observe a few things.  Player 1 can never make the same move twice in a row, otherwise he will lose (e.g.,, 1212).  Thus, on any turn, he has essentially two choices, whether to choose the symbol $p+1\pmod 3$ or $p-1\pmod 3$, where $p$ is the choice he made on his previous move.  We can thus describe the outcome of any game by a sequence of $+$'s and $-$'s representing which choice was made by Player 1 on each turn---without loss of generality, we assume his first move to be the symbol $0$.  For example, the game $0112201220$ corresponds to the sequence of choices $++-+$.  Note that the subsequences $--$ and $+-+$ both correspond to squares of nonsingleton blocks (e.g., $\underbracket[.1pt]{012}\underbracket[.1pt]{012}$, and $\underbracket[.1pt]{0112}\underbracket[.1pt]{0112}$, respectively).  Consequently, if the sequence of choices by Player 1 includes $+-$ as a subsequence, he will lose by the next turn.  Coupled with the fact that the sequence $+++++$ corresponds to a square (e.g., $\underbracket[.1pt]{011220}\underbracket[.1pt]{011220}$), this implies that the longest sequence of choices for Player 1 which does not produce consecutive identical blocks of length $\geq 2$ for this fixed strategy of Player 2 is $-++++-$, corresponding to a game with a total of 14 moves.
\end{proof}
A computer search has shown that the strategy described above for Player 2 is optimal, in the sense that Player 1 has a strategy which ensures that he will always survive until the 16th move of the game.

One can make many interesting modifications of Question \ref{q.cc}.  For example, what if we bias the game by restricting Player 2 to a subset of the symbols available to Player 1?  There does not seem to be an obvious way to take advantage of this kind of extra restriction in Local-Lemma based upper bounds.  And on the other side of things, the proof of Theorem \ref{t.cl} no longer works with this kind of restriction; in particular, in the case where Player 1 can choose from the symbols 0,1,2, and Player 2 can choose from the symbols 0,1, we have not ruled out the possibility that Player 1 can avoid the production of squares of nonsingletons indefinitely.

  There is, of course, the natural question coming from Theorem \ref{t.sblocks}:
\begin{q}
  What is the minimum $c$ required so that Player 1 has a strategy in the $c$-ary sequence game to avoid the appearance of consecutive identical blocks of total length $\geq 3$?
\end{q}
Our upper bound for this question is 64.  It is easy to see that 4 is a lower bound (give Player 2 the strategy of always choosing the digit 0, for example).

Another obvious direction of inquiry concerns the nonconstructive nature of our proofs.  For example, regarding Theorem \ref{t.grain}:
\begin{q}
  Give an explicit strategy for Player 1 to avoid consecutive repetition of blocks of length $\geq 2$ in the $c$-ary sequence game for some $c$.
\label{q.cgrain}
\end{q}
Note that Question \ref{q.cgrain} is not just interesting from the standpoint of proof techniques: the strategies guaranteed to exist by our theorems need not have finite description, thus it is natural to wonder if finite strategies do exist for these games.   We \emph{can} give explicit strategies for Player 1 in the situation where he always is allowed to know Player 2's next move in advance:
\begin{theorem}
  If Player 1 can always know Player 2's next move in advance, he has a strategy (with a finite description) in the $16$-ary sequence game which avoids any consecutive repetition of blocks of length $\geq 2$.  In the 48-ary sequence game, he has a strategy to avoid consecutive repetition altogether (even of blocks of length 1).
\end{theorem}
\noindent (Note that a Local-Lemma based argument can also show the existence of a strategy to avoid \emph{all} consecutive repetitions in the case where Player 1 always knows Player 2's next move.)
\begin{proof}
  For the first part of the theorem, we will let the $16$-ary sequence game be played using the 16 symbols $(a,b)$, $0\leq a,b\leq 3$.  Thus a play of the game results in a sequence of pairs $(a_0,b_0),(a_1,b_1),(a_2,b_2),(a_3,b_3),\dots,$ with $0\leq a_k,b_k\leq 3$ for all $k$.  Since Player 1 always knows Player 2's next move in advance, he can make his moves so that the sequence $\alpha_j=a_{j}-b_{j-1} \pmod 4$ is any sequence over $\Z_4$ of his choosing.  (Observe that, in the usual case where he cannot see future moves of Player 2, he could only control the even terms of this sequence.)  Observe now that if there is a square in the sequence resulting from the game:
\[
\dots,(a_{k+1},b_{k+1}),\dots,(a_{k+m},b_{k+m}),(a_{k+m+1},b_{k+m+1}),\dots,(a_{k+2m},b_{k+2m}),\dots,
\]
where $(a_{k+m+\ell},b_{k+m+\ell})=(a_{k+\ell},b_{k+\ell})$ for all $1\leq \ell\leq m$, then we have that the blocks $[\alpha_{k+2}, \alpha_{k+3},\dots,\alpha_{k+m}]$ and $[\alpha_{k+m+2}, \alpha_{k+m+3},\dots,\alpha_{k+2m}]$ from the sequence $\{a_j\}$ are identical.  Note that this is not a square, as the blocks are separated by the element $\alpha_{k+m+1}$.  However, Thue constructed in \cite{t2} a sequence over 4 symbols in which any identical blocks are separated by at least 2 symbols. (In general, he constructed $c$-ary sequences where identical blocks are separated by at least $c-2$ symbols, a result greatly improved upon by Dejean's conjecture for the values for which it has been confirmed). By making his moves so that $\{\alpha_j\}$ will be this sequence constructed by Thue, Player 1 can avoid any repetition of consecutive blocks of lengths at least 2 in the game.

To avoid all repetition in the 48-ary sequence game, Player 1 considers the symbols to consist of three classes of pairs $(a,b)$, $0\leq a,b\leq 3$, and plays as above, but simultaneously ensuring that no consecutive symbols are from the same class.
\end{proof}
Regarding Question \ref{q.cgrain}, it may seem ambitious to hope for explicit strategies in cases like the generalization of Beck's theorem (Theorem \ref{t.2game}), where no explicit construction is known even without the presence of the game. For the question on nonrepetitive $c$-ary sequences, however, the question seems quite natural.

\begin{q}
  Which patterns can Player 1 avoid in $c$-ary sequence games?
\end{q}
Recall that for games, pattern matching only counts if the matching partition contains only blocks of lengths $\geq 2$ (so, for example, Theorem \ref{t.grain} implies that Player 1 can avoid the pattern $xx$ in the $37$-ary game).  With this interpretation of what it means to `avoid' a pattern in a sequence game, one intriguing possibility is that the set of unavoidable patterns for sequence games is the same as the set of unavoidable patterns for sequences, which, as discussed in Section \ref{s.patt}, were characterized by Zimin and Bean et al. \cite{Zunav,bean}.
Note that the restriction that substituted words have length $\geq 2$ is not so unnatural, since in the game-free case, restrictions like this do not affect the set of unavoidable words (although they do affect the smallest base at which one can avoid a pattern).  Note that there is another natural question on patterns and games: is there some upper bound on the `game-avoidability index' of avoidable patterns?

This is similar to a question of Grytczuk in \cite{patgraph} regarding patterns avoidable on graphs of bounded maximum degree.  In that case also, all patterns without isolated variables are known to be avoidable.

\bibliographystyle{abbrv}

\end{document}